\documentclass[11pt,leqno]{amsart}
\usepackage[letterpaper, total={6in, 8in}, left=1.25in, top=1.5in]{geometry}
\usepackage{amsmath}
\usepackage{amsfonts}
\usepackage{amssymb}
\usepackage{graphicx}
\usepackage{color}
\usepackage{hyperref}
\usepackage{xcolor}

\newtheorem{theorem}{Theorem}[section]
\newtheorem*{theorem*}{Theorem}
\newtheorem{lemma}{Lemma}[section]
\newtheorem{corollary}[theorem]{Corollary}

\newtheorem{prop}{Proposition}[section]
\newtheorem{definition}[theorem]{Definition}
\newtheorem{example}[theorem]{Example}

\newtheorem{conjecture}[theorem]{Conjecture}

\newtheorem{remark}[theorem]{Remark}

\def\l{\lambda}

\def\p{\partial}

\def\R{\mathbb{R}}

\def \p {\partial}

\newcommand{\Hn}{\mathbb H}
\newcommand{\dd}{\,\mathrm d}
\numberwithin{equation}{section}

\begin{document}

\title[Faber-Krahn  inequalities for weighted Robin eigenvalues]{Faber-Krahn  inequalities for weighted Robin eigenvalues}

\author{Daguang Chen} 
\address{Department of Mathematical Sciences,
Tsinghua University,
Beijing 100084,
China}
\email{dgchen@tsinghua.edu.cn}

\author{Kui Wang} 
\address{School of Mathematical Sciences, Soochow University, Suzhou, 215006, China}
\email{kuiwang@suda.edu.cn}

\author{Anqiang Zhu}
\address{School of Mathematics and Statistics, Wuhan University, Wuhan 430072, China}
\thanks{Corresponding author. Email: \href{mailto:aqzhu.math@whu.edu.cn}{aqzhu.math@whu.edu.cn} (A. Zhu).}
\email{aqzhu.math@whu.edu.cn}

\begin{abstract}
In this paper, we study the isoperimetric  inequalities for the Robin eigenvalues of the weighted Laplacian with positive Robin parameter in the Euclidean space $\mathbb{R}^n$ and  the hyperbolic space $\mathbb{H}^n$, respectively. More precisely, we prove that among all bounded Lipschitz domains with fixed weighted volume, the geodesic ball centered at the origin minimizes the first Robin eigenvalue of the weighted Laplacian, provided that the Robin parameter and the radial log-convex density satisfy suitable conditions. Furthermore, we show that the second Robin eigenvalue is bounded below by the first Robin eigenvalue of the geodesic ball centered at the origin with half the weighted volume. Our results extend classical Faber-Krahn  inequalities to the setting of weighted spaces with log-convex densities.  We also derive a lower
bound for the second Robin eigenvalue in terms of the first
eigenvalue of the centered ball with half the weighted volume.
\end{abstract}

\subjclass[2020]{35P15, 58J50, 49Q20}

\keywords{First Robin eigenvalue, Weighted measure, Isoperimetric inequality, Faber-Krahn  inequality}

\maketitle

\section{Introduction}
Let $\Omega$ be a bounded domain with Lipschitz boundary in  the Euclidean space $\R^n$ or  the hyperbolic space $\mathbb{H}^n$. We consider the Robin eigenvalue problem for the weighted Laplacian
\begin{align}\label{differential equation}
    \begin{cases}
        -\mathrm{div} \big(e^{h(r(x))}\nabla u\big)=\lambda(\Omega,\beta)\, e^{h(r(x))} u,  & \text{ in } \Omega, \\
        \dfrac{\partial u}{\partial \nu}+\beta u=0, & \text{ on } \p \Omega,
    \end{cases}
\end{align}
where $r(x)=d(o,x)$ is the  distance function from the origin $o$, $h$ is a smooth even function, $\nu$ denotes the outward unit normal to $\p \Omega$, and $\beta\in \R$ is the Robin parameter. The weight $e^{h(r)}$ defines a radially log-convex density on the space.
The eigenvalues of \eqref{differential equation}, denoted by $\l_{k}(\Omega, \beta)$, are increasing and continuous in $\beta$,  and satisfy
\begin{align*} 
\l_{1}(\Omega, \beta)<\l_{2}(\Omega, \beta ) \le \l_{3}(\Omega, \beta) \le \cdots \rightarrow \infty,
\end{align*}
where each eigenvalue is counted with multiplicity. The first eigenfunction is positive. Moreover the theory of self-adjoint operators yields variational characterizations of Laplacian eigenvalues, and the first  eigenvalue of \eqref{differential equation} admits the following variational characterization:
\begin{align}\label{variational representation}
    \lambda_{1}(\Omega, \beta)=\inf_{u\in H^{1}(\Omega)\setminus\{0\}}\frac{\int_{\Omega}|\nabla u|^{2}e^{h(r(x))}\,d\mu + \beta\int_{\partial \Omega}u^{2}e^{h(r(x))}\,d\sigma}{\int_{\Omega}u^{2}e^{h(r(x))}\,d\mu}.
\end{align}
where $d\mu$ is the volume measure in $\R^n$ or $\mathbb{H}^n$, and  $d\sigma$ is the induced surface measure on $\p \Omega$.

The classical Faber-Krahn  inequality asserts that among all bounded domains of fixed volume, the ball uniquely minimizes the first Dirichlet eigenvalue of the Laplace operator. This result has been extended in various directions, including different boundary conditions, nonlinear operators, and Riemannian settings. For the Robin Laplacian with positive Robin parameter, the first Faber–Krahn inequality was established by Bossel \cite{BosselCR, BosselZAMP}  in two dimensions, and later generalized to arbitrary dimensions by Daners \cite{DanersMathann}. For nonlinear operators, the analogous result was obtained independently by Bucur-Daners~\cite{BucurCV} and Dai-Fu~\cite{Dai2011}. In the Riemannian setting, Chen-Cheng-Li \cite{ChendaguangJDE} proved that the Faber–Krahn inequality holds for the Robin Laplacian on bounded domains in manifolds with Ricci curvature bounded below, as well as in hyperbolic space.
For negative values of $\beta$, it was conjectured by Bareket \cite{Bar77} in 1977 that the ball would be the maximizer among domains in $\mathbb{R}^n$ with the same volume. However, in 2015, Freitas and  Krej\v{c}i\v{r}{i}k \cite{FK15} disproved Bareket's conjecture by showing that the ball is not a maximizer for sufficiently negative values of $\beta$. In the same paper, the authors showed that in dimension two, the disk uniquely maximizes $\l_{1}(\Omega, \beta)$ for $\beta <0$ with $|\beta|$ sufficiently small, and conjectured that the maximizer still has radial symmetry whenever $\beta<0$ and should switch from a ball to a shell at some critical value of $\beta$. In the present paper, we mainly focus on the case of positive Robin parameter.

In recent years, spectral inequalities on weighted manifolds have attracted considerable attention. Chiacchio and Gavitone \cite{Chiacchiomathann} established a Faber-Krahn inequality for the Hermite operator with Robin boundary conditions, which corresponds to the Laplacian in Gaussian space. Their work revealed that the optimal domain for the weighted problem differs from the Euclidean case, motivating further investigation of eigenvalue problems under log-convex measures. 
It is well known that the isoperimetric inequality serves as an essential tool in the study of  Faber-Krahn  inequality. The radial log-convex density conjecture first stated by
Kenneth Brakke and was recorded as Conjecture~3.12 in Rosales, Ca\~nete,
Bayle, and Morgan \cite{Morgan}; see also Morgan's historical note
\cite{MorganBlog2010}.

\begin{conjecture}\label{conj1}
In $\mathbb{R}^{n}$  equipped with a smooth, radial, log-convex density, balls centered at the origin are isoperimetric regions for any given volume.
\end{conjecture}
This conjecture has been extensively studied (see, e.g., \cite{Figalli13, Morgan}  and the references therein). Recently, Chambers \cite{ChamberJEMS}  resolved this  conjecture in the Euclidean setting, thereby providing a crucial foundation for the investigation of isoperimetric and spectral inequalities in weighted spaces. Subsequently, Silini \cite{Silini24}  extended the corresponding results to hyperbolic space. 
 For later use, set
\begin{equation}\label{eq:flat-core-radius}
	R_0:=\sup\{r\geq0:h(r)=h(0)\}\in[0,\infty].
\end{equation}

\begin{theorem}[\cite{ChamberJEMS, Silini24}]\label{thm:weighted-isoperimetry}
	The following weighted isoperimetric statements hold.
	\begin{enumerate}
		\item[\textup{(a)}]
		If \(h\colon\R\to\R\) is smooth, even, and convex, then centered
		balls minimize weighted perimeter at fixed weighted volume in
		\(\R^n\).  Up to null sets, every minimizer is either a ball centered
		at the origin or a Euclidean ball contained in the flat core
		\(B_{R_0}(0)\).
		\item[\textup{(b)}]
		If \(h\colon\R\to\R\) is smooth, even, and strictly convex, then
		geodesic balls centered at \(o\) are the unique weighted-isoperimetric
		regions of prescribed weighted volume in \(\Hn^n\).
	\end{enumerate}
\end{theorem}

Regarding Faber-Krahn  inequalities of weighted Laplacian with log-convex densities, Chen and Mao \cite{Maojing2025} established a series of such  inequalities for the Dirichlet eigenvalues of the weighted Laplacian on 
$\R^n$ and $\mathbb{H}^n$. Despite these developments, the Faber-Krahn inequality for the weighted Laplacian with Robin boundary conditions under radially log-convex measures remains largely unexplored. This paper aims to fill this gap. Our main results establish that, under suitable conditions on the Robin parameter $\beta$ and the radial log-convex density, the geodesic ball centered at the origin uniquely minimizes the first Robin eigenvalue among all bounded Lipschitz domains of fixed weighted volume. Moreover, we prove a corresponding inequality for the second Robin eigenvalue, showing that it is bounded below by the first Robin eigenvalue of the geodesic ball centered at the origin with half of the weighted volume.

More precisely, in the Euclidean space $\R^n$ equipped with a radially log-convex weight $e^{h(|x|)}$, we prove the following theorem.
\begin{theorem}\label{Euclidean log convex}
    Let $\Omega\subset \mathbb{R}^n$ be a bounded Lipschitz domain, and let $B(R)\subset \mathbb{R}^{n}$ be the ball centered at the origin with the same weighted volume as $\Omega$. Let
    \(h\colon\R\to\R\) be smooth, even, and convex, and let \(\beta>0\). Suppose that one of the following conditions holds:
    \begin{enumerate}
        \item[\textup{(i)}] $\beta > \frac{C_{1}(R)R}{2}$, where $C_{1}(R)=\sup_{r\in (0,R]} \left(h''(r) + \frac{h'(r)}{r}\right)$;
        \item[\textup{(ii)}] $0 \leq h''(r) \leq \frac{n-1}{r^2}$ for all $r \in (0,R)$.
    \end{enumerate}
    Then
   \begin{equation}\label{eq:Euclidean-FK}
   	\lambda_1(\Omega,\beta)
   	\geq \lambda_1(B(R),\beta).
   \end{equation}
   Equality holds if and only if, up to a null set, \(\Omega\) is either
   a ball centered at the origin or a Euclidean ball contained in
   \(B_{R_0}(0)\).  In particular, if \(h(r)>h(0)\) for every \(r>0\),
   then equality holds if and only if \(\Omega\) is a round ball centered at the origin up to a null set.
\end{theorem}

\begin{remark}
    For the log-convex Gaussian-type density \(e^{|x|^2/2}\),  we have $h(r) = r^2/{2}$, which yields $C_{1}(R)=2$. Hence, for $\beta>R$, the Faber-Krahn  inequality holds for this measure.
\end{remark}

In the hyperbolic space $\mathbb{H}^n$, we establish the analogous result.
\begin{theorem}\label{Hyperbolic space with log convex}
    Let $\Omega\subset \mathbb{H}^n$ be a bounded Lipschitz domain, and let $B(R)\subset \mathbb{H}^{n}$ be the geodesic ball centered at the origin with the same weighted volume as $\Omega$.  Let
    \(h\colon\R\to\R\) be smooth, even, and strictly convex, and let
    \(\beta>0\). 
    Suppose one of the following conditions holds:
    \begin{enumerate}
        \item[\textup{(i)}]  $\beta > \left(\frac{C_{2}(R)}{2}-\min\left\{1,\frac{n-1}{2}\right\}\right)\tanh R$, where $C_{2}(R)=\sup_{r\in (0,R]}\left(\frac{h'(r)}{\tanh r}+\frac{h''(r)}{\tanh' r}\right)$;
        \item[\textup{(ii)}] $0 \leq h''(r) \leq \frac{n-1}{\sinh^2 r}$ for all $r \in (0,R)$.
    \end{enumerate}
Then
\begin{equation}\label{eq:Hyperbolic-FK}
    	\lambda_1(\Omega,\beta)
    	\geq\lambda_1(B(R),\beta).
    \end{equation}
   with equality if and only if $\Omega$ is a geodesic ball centered at the origin up to a null set.

\end{theorem}

\begin{example}
    Li-Xu \cite{LihaizhongPAMS} and Xia-Scheuer \cite{XiachaoTAMS} studied the isoperimetric inequality on $\mathbb{H}^{n}$ with the weighted measure $\cosh r d\mu$, i.e. $h(r)=\log \cosh r$. Then  
    $$h'(r)=\frac{\sinh r}{\cosh r}=\tanh r, \qquad h''(r)=\frac{1}{\cosh^{2}r}\leq \frac{1}{\sinh^{2}r}.$$ 
    Then by Theorem \ref{Hyperbolic space with log convex}, the Faber-Krahn  inequality for Robin weighted Laplacian with density  $\cosh r d\mu$ on $\mathbb{H}^{n}$ holds for any   $\beta>0$.
\end{example}

In \cite{KennegyPAMS}, Kennedy proved that, among all bounded Lipschitz domains of fixed volume, the second Robin eigenvalue of the Laplacian  is minimized by the disjoint union of two equal-volume balls. In \cite{Maojing2025}, Chen and Mao investigated the second eigenvalue of the Witten Laplacian subject to Dirichlet boundary conditions on weighted $\mathbb{R}^n$ and $\mathbb{H}^n$. Building on these works, we establish the following theorem.
\begin{theorem}\label{thm3}
Let $\Omega$ be a bounded Lipschitz domain in $\R^n$ (or $\mathbb{H}^n$), and let $\lambda_2(\Omega,\beta)$ denote the second eigenvalue of \eqref{differential equation}. Assume that $\beta$ satisfies the conditions in Theorem \ref{Euclidean log convex}(or in Theorem \ref{Hyperbolic space with log convex}). For domains satisfying the first set of conditions, we further require that the function
\begin{align*}
    \frac{(n-1)}{r^2}-h''(r) \quad (\text{in $\R^n$} )\quad \text{or}\quad   \frac{(n-1)}{\sinh^2r}-h''(r) \quad (\text{in $\mathbb{H}^n$} )
\end{align*}
 is monotonically decreasing. Then
\begin{align}
    \lambda_2(\Omega,\beta)\ge \lambda_1(D,\beta),
\end{align}
where $D$ is the geodesic ball centered at the origin with  half the weighted volume of $\Omega$.
\end{theorem}

The main novelty of this work lies in establishing  complete Faber-Krahn inequalities for the Robin eigenvalues of the weighted Laplacian under radially log-convex measures on both Euclidean and hyperbolic spaces. Although the Faber-Krahn inequality for the Robin Laplacian in the unweighted setting is now well understood, the weighted Robin case presents substantial new challenges. For instance, in the unweighted case, the key technical tool is the monotonicity of 
$v(r):=-u'(r)/u(r)$, which ensures the validity of the level-set rearrangement. For general radially log-convex weights, this monotonicity does not hold automatically. The paper introduces two sufficient conditions-one involving a lower bound  on $\beta$, the other based on a pointwise  constraint on density function $h$-that rigorously restore the property 
$v(r)<v(R)$. This constitutes the central technical breakthrough. Naturally, since the present paper focuses primarily on the case of positive Robin parameter, it is equally natural to investigate the corresponding problem for negative Robin parameters in the setting of weighted measures. We plan to address this negative-parameter case in a future work.

The paper is organized as follows. In Section \ref{sect2}, we introduce the necessary notation and preliminary results on weighted Sobolev spaces and eigenfunctions. In Section \ref{sect3}, we derive a representation formula for the first eigenvalue and prove a key strict inequality lemma. Section \ref{sect4} is devoted to the analysis of eigenvalues on geodesic balls, including the radial monotonicity of eigenfunctions and the maximum principle for the associated function $v(r)$. In Section \ref{sect5}, we prove Theorems \ref{Euclidean log convex} and \ref{Hyperbolic space with log convex}. Finally, in Section \ref{sect6}, we prove Theorem \ref{thm3}.

\section{Preliminaries}\label{sect2}

Throughout this paper, let $M^n$ denote either the $n$-dimensional Euclidean space $\mathbb{R}^n$ or the $n$-dimensional hyperbolic space $\mathbb{H}^n$. Let $S(t)$ be the sine-type function defined
\begin{align*}
    S(t)=\begin{cases}
    t, &\quad \text{if} \quad M^n= \R^n,\\
     \sinh t, &\quad \text{if} \quad M^n=\mathbb{H}^n,
    \end{cases}
\end{align*}
and let  $C(t):=S'(t)$ denote the corresponding the cosine-type function.
Fix a  base point $o \in M^n$ (the origin). Given a positive density function $f$ on $M^n$, and a   set $A \subset M^n$ with locally finite perimeter, we define the weighted perimeter and weighted volume, respectively by
\begin{align*}
\operatorname{P}_{f}(A) = \int_{\partial^* A} f \, d\sigma, \qquad
\operatorname{Vol}_{f}(A) = \int_{A} f \, d\mu,
\end{align*}
where $d\mu$ denotes the $n$-dimensional standard Riemannian volume meausre in $\R^n$ or $\mathbb{H}^n$, $d\sigma$ is the induced $(n-1)$-dimensional Hausdorff measure on the reduced boundary  $\partial^* A$.

A density function $f \colon M^n \to \mathbb{R}_{>0}$ is said to be \emph{(strictly) radially log-convex} if  it can be written as $$f(x)=e^{h(r(x))}$$ and
 $h \colon \mathbb{R} \to \mathbb{R}$ is  smooth, (strictly) convex, and even.
Where $r(x) = d(o,x)$ is the  distance function from the orgin $o$,  In  what follows, we denote the weighted measure  with density  $e^{h(r(x))} $ by $$d\gamma = e^{h(r(x))} \, d\mu.$$

Let $\Omega \subset M^n$ be a bounded domain. We  denote by $L^2(\Omega, \gamma)$ the space of all real-valued measurable functions on $\Omega$ such that
\begin{align*}
\|u\|_{L^2(\Omega, \gamma)}^2 := \int_{\Omega} u^2(x) \, d\gamma< +\infty.
\end{align*}
The associated weighted Sobolev space is defined as
\begin{align*}
H^1(\Omega, \gamma) := \left\{ u \in W^{1,2}_{\mathrm{loc}}(\Omega) : u, |\nabla u| \in L^2(\Omega, \gamma) \right\},
\end{align*}
 equipped with the norm
\begin{align*}
\|u\|_{H^1(\Omega, \gamma)} = \|u\|_{L^2(\Omega, \gamma)} + \|\nabla u\|_{L^2(\Omega, \gamma)}.
\end{align*}
Since $\Omega$ is bounded and $h$ is smooth,  there exist positive constants $c_1$ and $c_2$ such that
\begin{align*}
c_1 \le e^{h(r(x))} \le c_2 \quad \text{for all } x \in \Omega.
\end{align*}
Consequently, the weighted $L^2$-norm $\|\cdot\|_{L^2(\Omega,\gamma)}$ is equivalent to the standard $L^2(\Omega)$ norm. In particular, the weighted Sobolev space 
 $H^1(\Omega, \gamma)$ coincides with the usual Sobolev space $H^1(\Omega)$ as a set, and their norms are equivalent.

\begin{lemma}\label{existence and positive}
   Assume $\beta>0$. Then $\lambda_{1}(\Omega,\beta) > 0$, and the infimum in \eqref{variational representation} is attained by a strictly positive eigenfunction.
\end{lemma}

\begin{proof}
This follows from a standard compactness argument, and for completeness we give a proof here.
Let $$\mathcal{K} := \left\{ u \in H^1(\Omega,\gamma) : \int_{\Omega} u^2(x)  \, d\gamma = 1 \right\},$$
and define the energy functional $\Phi : \mathcal{K} \to \mathbb{R}$ by
\begin{align*}
\Phi(u) = \int_{\Omega} |\nabla u|^2 e^{h(r(x))} \, d\mu + \beta \int_{\partial\Omega} u^2 e^{h(r(x))} \, d\sigma.
\end{align*}
Let $\{u_i\} \subset \mathcal{K}$ be a minimizing sequence such that $\Phi(u_i) \to \lambda_{1}(\Omega,\beta)$. Since the weighted and standard Sobolev norms are equivalent, $\{u_i\}$ is bounded in the standard $H^1(\Omega)$. By the Rellich-Kondrachov compactness theorem and the compact trace embedding, there exists a subsequence (still denoted by $\{u_i\}$) and a function $u \in H^1(\Omega)$ such that $u_i \rightharpoonup u$ weakly in $H^1(\Omega)$, $u_i \to u$ strongly in $L^2(\Omega)$, and $u_i \to u$ strongly in $L^2(\partial\Omega)$. 

The strong convergences in $L^2$ imply that $u \in \mathcal{K}$. The weak lower semicontinuity of the Dirichlet integral and the strong convergence of the boundary trace yield
\begin{align}
    \Phi(u) \leq \liminf_{i \to \infty} \Phi(u_i) = \lambda_{1}(\Omega,\beta).
\end{align}
Since $u \in \mathcal{K}$, we also have $\lambda_{1}(\Omega,\beta) \leq \Phi(u)$. Thus, $\lambda_{1}(\Omega,\beta) = \Phi(u)$, meaning $u$ is a minimizer. 

Furthermore, $\lambda_1(\Omega, \beta) > 0$. Indeed, if $\lambda_1 = 0$, then $\Phi(u) = 0$, which implies $\nabla u = 0$ a.e. in $\Omega$ and $u = 0$ on $\partial\Omega$ (since $\beta > 0$). This would force $u \equiv 0$, contradicting $u \in \mathcal{K}$.

Finally, since $\Phi(|u|) = \Phi(u)$, we may assume without loss of generality that $u \geq 0$. By standard elliptic regularity and the strong maximum principle (see, e.g., \cite[Chapter 2, Theorem 5]{Protter}), we conclude that $u(x) > 0$ for all $x \in \Omega$.
\end{proof}

\begin{lemma}\label{lem:regularity}
    Let $\Omega$ be a bounded Lipschitz domain and $\psi$ be the first eigenfunction of \eqref{differential equation}. Then $\psi \in C(\overline{\Omega}) \cap C^1(\Omega)$.
\end{lemma}
\begin{proof}
    Since the weight $e^{h(r(x))}$ is smooth and strictly positive, standard interior elliptic regularity theory (e.g., \cite[Theorem 8.10]{Trudinger}) implies that $\psi \in C^1(\Omega)$. The global continuity $\psi \in C(\overline{\Omega})$ up to the boundary follows from the boundary regularity theory for elliptic equations on Lipschitz domains (see, e.g., \cite[pp. 466--467]{Ladyzhenskaya}).
\end{proof}
\begin{definition}\label{def:level_sets}
Let $\psi$ denote the positive first eigenfunction of $\Omega$, normalized such that $\|\psi\|_{L^\infty(\Omega)} = 1$. For each $t \in (0,1)$, we define the superlevel sets and their boundaries as follows:
\begin{align*}
U(t) &:= \{x \in \Omega : \psi(x) > t\}, \\
S(t) &:= \{x \in \Omega : \psi(x) = t\}, \\
\Gamma_{t} &:= \{x \in \partial\Omega : \psi(x) > t\}.
\end{align*}
\end{definition}
For any nonnegative measurable function $\varphi \colon \Omega \to [0,\infty)$, we define the functional
\begin{align}\label{eq:H_Omega}
H_{\Omega}(U(t),\varphi) := \frac{1}{|U(t)|_\gamma} \left( \int_{S(t)} \varphi e^{h(r(x))} \, d\sigma
+ \int_{\Gamma(t)} \beta e^{h(r(x))} \, d\sigma
- \int_{U(t)} \varphi^2(x) e^{h(r(x))} \, d\mu \right),
\end{align}
where 
$$
|U(t)|_{\gamma} := \int_{U(t)} e^{h(r(x))} \, d\mu
$$
denotes the weighted volume of $U(t)$.

\section{A Representation Formula for $\lambda_{1}$}\label{sect3}
In this section,  $\Omega$ is a bounded Lipschitz domain in  either $\mathbb{R}^n$ or $\mathbb{H}^n$, equipped with the radial weight $e^{h(r(x))}$.

\begin{prop}\label{lambda representation 3.1}
    Let $\psi>0$ be the positive first eigenfunction of \eqref{differential equation} corresponding to $\lambda_{1}(\Omega,\beta)$, normalized by $\|\psi\|_{L^\infty(\Omega)} = 1$. Then for every $t\in (0,1)$, we have
    \begin{align}\label{eq:lambda_rep}
        \lambda_{1}(\Omega,\beta) = H_{\Omega}\!\left(U(t), \frac{|\nabla \psi|}{\psi}\right).
    \end{align}
\end{prop}

\begin{proof}
For  $t\in (0, 1)$. Applying the divergence theorem to the vector field $\frac{\nabla \psi}{\psi} e^{h(r(x))}$ on the domain $U(t)$, we obtain
    \begin{align}\label{3.2}
        &\int_{S(t)}\frac{\nabla \psi}{\psi}\cdot \nu \, e^{h}\,d\sigma + \int_{\Gamma_{t}}\frac{\nabla \psi}{\psi}\cdot \nu \, e^{h}\,d\sigma \\
        =& \int_{U(t)}\operatorname{div}\!\left(\frac{\nabla \psi}{\psi}e^{h}\right) d\mu \nonumber \\
        =& \int_{U(t)}\left(\frac{\Delta \psi}{\psi}-\frac{|\nabla \psi|^{2}}{\psi^{2}}+\frac{\nabla \psi \cdot \nabla h}{\psi}\right)e^{h}\,d\mu. \nonumber \\
        =&-\lambda_{1}(\Omega,\beta)|U(t)|_{\gamma} - \int_{U(t)}\frac{|\nabla \psi|^{2}}{\psi^{2}}e^{h}\,d\mu,\nonumber
    \end{align}
    where in the last equality, we used the weighted eigenvalue equation $$-\Delta \psi - \nabla \psi \cdot \nabla h = \lambda_1(\Omega,\beta) \psi.$$
We now evaluate the boundary integrals. On the interior level set $S(t)$, the outward unit normal to $U(t)$ points in the direction of decreasing $\psi$, so $\nu = -\frac{\nabla \psi}{|\nabla \psi|}$. Hence
    \begin{align}\label{3.3}
        \int_{S(t)}\frac{\nabla \psi}{\psi}\cdot \nu \, e^{h}\,d\sigma = -\int_{S(t)}\frac{|\nabla\psi|}{\psi}e^{h}\,d\sigma.
    \end{align}
On the boundary portion $\Gamma(t)$, $\nu$ is the outward unit normal to $\Omega$. The Robin boundary condition $\frac{\partial \psi}{\partial \nu} + \beta \psi = 0$, gives$\frac{\nabla \psi}{\psi} \cdot \nu = -\beta$, and therefore
    \begin{align}\label{3.4}
        \int_{\Gamma_{t}}\frac{\nabla \psi}{\psi}\cdot \nu \, e^{h}\,d\sigma = -\int_{\Gamma_{t}}\beta \, e^{h}\,d\sigma.
    \end{align}
Substituting \eqref{3.3} and \eqref{3.4} into \eqref{3.2}, we get
\begin{align}\label{eq:lambda_rep_expanded}
    \lambda_{1}(\Omega,\beta) &= \frac{1}{|U(t)|_{\gamma}}\left(\int_{S(t)}\frac{|\nabla \psi|}{\psi} e^{h}\,d\sigma + \int_{\Gamma_{t}}\beta \, e^{h}\,d\sigma - \int_{U(t)}\frac{|\nabla \psi|^{2}}{\psi^{2}}e^{h}\,d\mu\right) \\
    &= H_{\Omega}\!\left(U(t), \frac{|\nabla \psi|}{\psi}\right), \nonumber 
\end{align}
which completes the proof.
\end{proof}

\begin{prop}\label{Key prop}
Let $\varphi:\Omega\to [0,\infty)$ be a function in $L^{2}(\Omega,\gamma)$, and let $\psi>0$ be the positive first eigenfunction corresponding to $\lambda_{1}(\Omega,\beta)$. Define
\begin{align}
    w := \varphi - \frac{|\nabla \psi|}{\psi}, \qquad F(t) := \int_{U(t)}w\,\frac{|\nabla\psi|}{\psi}e^{h}\,d\mu.
\end{align}
Then $F:(0,1)\to \mathbb{R}$ is locally absolutely continuous. Moreover, for almost every $t\in(0,1)$, the following inequality holds:
\begin{align}\label{key inequality}
    H_{\Omega}(U(t),\varphi) \leq \lambda_{1}(\Omega,\beta) - \frac{1}{t\,|U(t)|_\gamma}\frac{d}{dt}\big(t^{2}F(t)\big).
\end{align}
Furthermore, if the set where $\varphi$ differs from $\frac{|\nabla\psi|}{\psi}$ has positive weighted measure in $U(t)$, i.e.,
\begin{align*}
\left|\left\{x\in U(t): \varphi(x)\neq \frac{|\nabla\psi(x)|}{\psi(x)}\right\}\right|_{\gamma} > 0,
\end{align*}
then inequality \eqref{key inequality} is strict.
\end{prop}

\begin{proof}
Using the elementary algebraic inequality $a^2 - b^2 \ge 2(a-b)b$, we have
\begin{align*}
\varphi^2 - \frac{|\nabla\psi|^{2}}{\psi^{2}} \ge 2 \left(\varphi - \frac{|\nabla\psi|}{\psi}\right)\frac{|\nabla\psi|}{\psi} = 2w \frac{|\nabla\psi|}{\psi}.
\end{align*}
Combining this with  \eqref{eq:lambda_rep_expanded}, we obtain
\begin{align}\label{eq:H strict}
    H_{\Omega}(U(t),\varphi) - \lambda_{1}(\Omega,\beta)
    &= \frac{1}{|U(t)|_\gamma}\left(\int_{S(t)}w \, e^{h}\,d\sigma - \int_{U(t)}\left(\varphi^{2}-\frac{|\nabla\psi|^{2}}{\psi^{2}}\right)e^{h}\,d\mu\right) \\
    &\leq \frac{1}{|U(t)|_\gamma}\left(\int_{S(t)}w \, e^{h}\,d\sigma - 2\int_{U(t)}w\frac{|\nabla \psi|}{\psi} e^{h}\,d\mu\right).\nonumber 
\end{align}
If $\varphi \neq \frac{|\nabla\psi|}{\psi}$ on a subset of $U(t)$ with positive weighted measure, then the above inequality is strict.

Next, since $\|\psi\|_{L^\infty} = 1$, the  co-area formula gives
\begin{align*}
    F(t) &= \int_{U(t)}w\frac{|\nabla \psi|}{\psi}e^{h}\,d\mu 
    = \int_{t}^{1}\left(\int_{S(\tau)}\frac{w}{\psi}e^{h}\,d\sigma\right)d\tau 
    = \int_{t}^{1}\frac{1}{\tau}\left(\int_{S(\tau)}w \, e^{h}\,d\sigma\right) d\tau.
\end{align*}
This representation shows that $F(t)$ is locally absolutely continuous on $(0,1)$. Differentiating $t^2 F(t)$ with respect to $t$, we get
\begin{align*}
    \frac{d}{dt}\left(t^{2}F(t)\right) &= 2tF(t) + t^{2}F'(t) \\
    &= 2t\int_{U(t)}w\frac{|\nabla\psi|}{\psi}e^{h}\,d\mu - t^{2}\left(\frac{1}{t}\int_{S(t)}w \, e^{h}\,d\sigma\right) \\
    &= -t\left(\int_{S(t)}w \, e^{h}\,d\sigma - 2\int_{U(t)}w\frac{|\nabla \psi|}{\psi}e^{h}\,d\mu\right).
\end{align*}
Substituting this  into \eqref{eq:H strict} yields
\begin{align*}
    H_{\Omega}(U(t),\varphi) - \lambda_{1}(\Omega,\beta) \leq -\frac{1}{t|U(t)|_\gamma}\frac{d}{dt}\big(t^{2}F(t)\big),
\end{align*}
which proves \eqref{key inequality}.
\end{proof}

\begin{prop}\label{lambda low bound}
Let $\varphi:\Omega\to [0,\infty)$ be a measurable function in $L^2(\Omega,\gamma)$. Assume that the set $\left\{x\in\Omega: \varphi(x)\neq \frac{|\nabla \psi(x)|}{\psi(x)}\right\}$ has positive weighted measure. Then there exists a subset $S\subset(0,1)$ of positive Lebesgue measure such that
\begin{align}\label{3.9}
    \lambda_{1}(\Omega,\beta) > H_{\Omega}(U(t),\varphi)
\end{align}
for all $t\in S$.
\end{prop}

\begin{proof}
Suppose,  to the contrary, that $\lambda_{1}(\Omega,\beta) \leq H_{\Omega}(U(t),\varphi)$ holds for almost every $t\in(0,1)$. Then  Proposition \ref{Key prop} implies
\begin{align*}
-\frac{1}{t|U(t)|_\gamma}\frac{d}{dt}\big(t^{2}F(t)\big) \geq 0 \quad \text{for a.e. } t\in(0,1).
\end{align*}
Define $G(t) := t^2 F(t)$. Since $F$ is locally absolutely continuous, so is $G$. The above inequality implies that $G'(t) \le 0$ a.e., hence $G$ is non-increasing on $(0,1)$. Since $U_1 = \emptyset$, we have $F(1)=0$ and thus $G(1)=0$. The monotonicity of $G$ then implies $G(t) \geq 0$ for all $t\in(0,1)$.

By  assumption, $\varphi \not\equiv \frac{|\nabla\psi|}{\psi}$ on a subset of $\Omega$ with positive weighted measure. Since $\bigcup_{t>0}U(t)=\Omega$, there exists $t_0\in(0,1)$ such that
\begin{align*}
\left|\left\{x\in U(t): \varphi(x)\neq \frac{|\nabla\psi(x)|}{\psi(x)}\right\}\right|_{\gamma} > 0
\end{align*}
for all $t\in(0,t_0)$. Consequently, the inequality in Proposition \ref{Key prop} is strict for $t\in(0,t_0)$, which yields $G'(t) < 0$ a.e. on $(0, t_0)$. Hence 
\begin{align*}
\lim_{t\to 0^+}G(t) \ge G(t_0) > G(1) = 0.
\end{align*}

On the other hand, from the definitions of $F(t)$ and $w$, and using  $\psi(x) > t$ for $x \in U(t)$, we have
\begin{align*}
F(t)
&= \int_{U(t)}\left(\varphi - \frac{|\nabla \psi|}{\psi}\right)\frac{|\nabla \psi|}{\psi}e^{h}\,d\mu \\
&\leq \int_{U(t)}\varphi \frac{|\nabla \psi|}{\psi}e^{h}\,d\mu \\
&\leq \frac{1}{t}\int_{U(t)}\varphi |\nabla \psi|e^{h}\,d\mu.
\end{align*}
By the  Cauchy-Schwarz inequality, since  $\varphi \in L^2(\Omega,\gamma)$ and $\psi \in H^1(\Omega,\gamma)$, there exists a constant $C$, independent of $t$, such that 
\begin{align*}
F(t)
&\leq \frac{1}{t}\left(\int_{U(t)}\varphi^2 e^{h}\,d\mu\right)^{\frac12}\left(\int_{U(t)}|\nabla\psi|^2 e^{h}\,d\mu\right)^{\frac12}
\leq \frac{C}{t}.
\end{align*}
Therefore
\begin{align*}
0 < \lim_{t\to 0^+}G(t) = \lim_{t\to 0^+}t^2 F(t) \leq \lim_{t\to 0^+}t^2 \left(\frac{C}{t}\right) = \lim_{t\to 0^+} Ct = 0,
\end{align*}
which is a contradiction. This completes the proof.
\end{proof}

\section{Eigenvalues on geodesic balls centered at the origin}\label{sect4}
Let $B(R)$ denote  the geodesic ball of radius $R$ centered at the origin $o$  in either $\R^n$ or $\mathbb{H}^n$. By the radial symmetry of both the ball  and  the weight function $e^{h(r)}$, the first eigenfunction of \eqref{differential equation} for $B(R)$ is radially symmetric. Consequently, the first eigenvalue equations reduces to the following one-dimensional boundary value problem:
\begin{align}\label{4.1}
    \begin{cases}
   u''(r)+\left((n-1)\frac{S'(r)}{S(r)}+h'(r)\right)u'(r)+\lambda_1(B(R), \beta) u(r)=0, \quad r\in (0, R),\\
   u'(0)=0, \quad u'(R)=-\beta u(R),
\end{cases}
\end{align}
where $S(r)=r$  in the Euclidean case $\R^n$ and $S(r)=\sinh r$ in the hyperbolic case $\mathbb{H}^n$. For the eigenfunction $u(r)$, we have the following monotonicity lemma.
\begin{lemma}\label{lem:radial_decreasing}
Suppose $\beta>0$. Then $u'(r)<0$ for all $r\in (0, R)$; hence $u(r)$ is strictly decreasing on $(0, R)$.
\end{lemma}

\begin{proof}
From   \eqref{4.1}, it follows that
\begin{align*}
   \frac{d}{dr}\left(S^{n-1}(r) e^{h}\frac{du}{dr}\right) = -\lambda_1(B(R), \beta) S^{n-1}(r) e^h u < 0,
\end{align*}
where the inequality holds because 
$\lambda_1(B(R), \beta)>0$ by Lemma \ref{existence and positive} and $u(r)>0$. Therefore, for all $r\in (0, R)$,
 \begin{align*}
   S(r)^{n-1} e^{h(r)}\frac{du}{dr}<\lim_{r\to 0^+} S(r)^{n-1} e^h u'(r)=0, 
 \end{align*}
 where we used $S(0)=0$ and the boundary condition $u'(0)=0$. Therefore we conclude
 $$
 u'(r)<0
 $$
 for all $r\in (0, R)$.
\end{proof}
Define $v(r) = -\frac{u'(r)}{u(r)}$. Then 
\begin{equation}\label{eq:v-origin}
	v(0)=0,\qquad v'(0)=\frac{\lambda}{n}>0,\qquad v(R)=\beta.
\end{equation}

\subsection{The Euclidean Case}
Restricting to the ball $B(R)\subset\R^n$, equation \eqref{4.1} becomes
\begin{align}\label{1 dimen eigen R}
    \begin{cases}
   u''(r)+\left(\frac{n-1}{r}+h'(r)\right)u'(r)+\lambda_1(B(R), \beta) u(r)=0, \quad r\in (0, R),\\
   u'(0)=0, \quad u'(R)=-\beta u(R).
\end{cases}
\end{align}
 By Lemma \ref{lem:radial_decreasing}, $v(r) > 0$ for all $r\in (0,R)$. 
A direct computation gives the first-order differential equation for $v$:
\begin{align}\label{v first derivative}
    v'(r) = -\left(\frac{n-1}{r}+h'(r)\right)v(r) + v(r)^{2} + \lambda_1(B(R), \beta),
\end{align}
and differentiating once more gives the second-order equation:
\begin{align}\label{v second derivative}
    v''(r) = -\left(\frac{n-1}{r}+h'(r)\right)v'(r) - \left(h''(r)-\frac{n-1}{r^{2}}\right)v(r) + 2v(r)v'(r).
\end{align}

\begin{prop}\label{prop:Euclidean_v_max}
Let $\beta > 0$, and suppose that $h(r)$ satisfies one of the following two conditions: 
\begin{enumerate}
    \item[\textup{(i)}] $\beta > \frac{C_{1}(R)R}{2}$, where $C_{1}(R)=\sup_{r\in (0,R]} \left(h''(r) + \frac{h'(r)}{r}\right)$;
    \item[\textup{(ii)}] $0\leq h''(r)\leq \dfrac{n-1}{r^2}$ for all $r\in(0,R)$.
\end{enumerate}
Then
\begin{equation}\label{eq:v-Euclidean-bound}
	0\leq v(r)<v(R)=\beta
	\qquad(0\leq r<R).
\end{equation}
\end{prop}

\begin{proof}
Since \(h\) is even and convex,
\[
h'(0)=0,\qquad h'(r)\geq0,\qquad h''(r)\geq0
\quad(r\geq0).
\]
Taylor's theorem gives
\(\displaystyle\lim_{r\to 0}h'(r)/r=h''(0)\).

Assume first condition~\textup{(ii)} and put
\[
B(r):=\frac{n-1}{r^2}-h''(r)\geq0.
\]
For \(w=v'\), equation
\eqref{v second derivative} becomes
\[
w'+\left(\frac{n-1}{r}+h'-2v\right)w=Bv\geq0.
\]
By \eqref{eq:v-origin}, \(w(r)>0\) for all sufficiently small
\(r>0\).  Fix such an \(\varepsilon\), and define
\[
\mu(r):=
\exp\left(
\int_\varepsilon^r
\left(\frac{n-1}{s}+h'(s)-2v(s)\right)\dd s
\right)>0.
\]
Then
\[
(\mu w)'=\mu Bv\geq0
\]
and, for \(r\geq\varepsilon\),
\[
\mu(r)w(r)
=
w(\varepsilon)+\int_\varepsilon^r\mu(s)B(s)v(s)\dd s>0.
\]
Letting \(\varepsilon\) vary shows that \(v'(r)>0\) on \((0,R)\),
and \eqref{eq:v-Euclidean-bound} follows.

We now turn to case (i):
 $\beta >\frac 1 2 C_1(R) R$, where
 $$C_1(R):= \sup_{r\in (0,R]} \left(h''(r) + \frac{h'(r)}{r}\right), $$
 which is a nonnegative constant, due to  that $h(r(x))$ is smooth,  $h'(0)=0$, and $h''(0)\ge 0$.
 
Define the auxiliary function $$\xi(r) = : v(r) - \frac 1 2 C_1(R) r.$$   Substituting   $v(r) = \xi(r) + \frac 1 2 C_1(R) r$ into \eqref{v second derivative} and rearranging terms, we obtain
\begin{align}\label{4.5}   
0 = \xi'' + \left(\frac{n-1}{r}+h'-2v\right)\xi' - \left(C_1(R)-h''+\frac{n-1}{r^{2}}\right)\xi + \frac{C_1(R)}{2}\left(h'+h''r-C_1(R)r\right).
\end{align}

By the definition of $C_1(R)$, we have for all $r\in (0, R)$
\begin{align}\label{4.6}
h'(r)+h''(r)r-C_1(R)r = r\left(\frac{h'(r)}{r}+h''(r)-C_1(R)\right) \leq 0,
\end{align}
and 
\begin{align}\label{4.7}
 C_1(R) - h''(r) + \frac{n-1}{r^2} \ge \frac{h'(r)}{r} + \frac{n-1}{r^2} > 0.
\end{align}
By the assumption $\beta > kR$,  boundary condition $v(R)=\beta$, and $v(0)=0$, we have
\begin{align*}
   \xi(R) = \beta - \frac {C_1(R)}2R > 0, \quad \text{and}\quad \xi(0) = v(0) - 0 = 0. 
\end{align*}
Suppose $\xi(r)$ attains its global maximum on $[0,R]$ at an interior point $r_{0}\in(0,R)$. Then $\xi'(r_{0})=0$, $\xi''(r_{0})\leq 0$, and $\xi(r_0) > 0$. Then from \eqref{4.5}, it follows 
\begin{align*}
\xi''(r_{0}) = \left(C_{1}(R)-h''(r_0)+\frac{n-1}{r_0^{2}}\right)\xi(r_0) - \frac{C_{1}(R)}{2}r_0\left(\frac{h'(r_0)}{r_0}+h''(r_0)-C_{1}(R)\right).
\end{align*}
Using inequalities \eqref{4.6} and \eqref{4.7}, we have
$$
\xi''(r_0)>0,
$$
contradicting with $\xi''(r_0)\le 0$. Therefore, the maximum of $\xi(r)$
 the maximum of $\xi$ must occur at the boundary $r=R$. Hence $\xi(r) < \xi(R)$ for all $r\in(0,R)$, which implies
\begin{align*}
v(r) < v(R) - \frac{C_{1}(R)}{2}(R-r) < v(R).
\end{align*}
This completes the proof.
\end{proof}

\begin{remark}
	The lower bound on \(\beta\) in
	Proposition~\ref{prop:Euclidean_v_max}\textup{(i)} cannot in general be
	discarded. 
For the Gaussian weight $h(r) = \frac{1}{2}r^2$, the function $\varphi(r)=e^{-\frac{1}{2}r^{2}}$ solves  \eqref{1 dimen eigen R} with eigenvalue $\lambda=n$ and boundary parameter $\beta=R$. By the variational characterization of the principal eigenvalue, for any $\beta < R$,
\begin{align*}
n =& \frac{\int_{B(R)}|\nabla \varphi|^{2}e^{\frac{1}{2}r^{2}} + R\int_{\partial B(R)}\varphi^{2}e^{\frac{1}{2}r^{2}}}{\int_{B(R)}\varphi^{2}e^{\frac{1}{2}r^{2}}} \\
\geq& \frac{\int_{B(R)}|\nabla \varphi|^{2}e^{\frac{1}{2}r^{2}} + \beta\int_{\partial B(R)}\varphi^{2}e^{\frac{1}{2}r^{2}}}{\int_{B(R)}\varphi^{2}e^{\frac{1}{2}r^{2}}} \\
>& \lambda_1(B(R),\beta).
\end{align*}
From  \eqref{v first derivative}, we have
\begin{align*}
v'(r) = -\left(\frac{n-1}{r}+r\right)v(r)+v^2(r)+\lambda_1(B(R),\beta) \leq -rv(r)+v^2(r)+n.
\end{align*}
Evaluating at $r=R$ where $v(R)=\beta$, we find $v'(R) < 0$ provided that $\beta$ lies between the roots of $-R\beta + \beta^2 + n = 0$, i.e.,
\begin{align*}
\frac{R-\sqrt{R^{2}-4n}}{2} < \beta < \frac{R+\sqrt{R^{2}-4n}}{2}.
\end{align*}
This illustrates that if $\beta$ is too small, $v(r)$ may attain its maximum in the interior of $(0,R)$. This example shows that a
size condition on \(\beta\) is genuinely needed for the boundary
maximum argument, although it does not assert that the threshold in
Proposition~\ref{prop:Euclidean_v_max}\textup{(i)} is optimal.
\end{remark}

\subsection{The Hyperbolic Case}
On the geodesic ball $B(R)\subset \mathbb{H}^{n}$,  equation \eqref{4.1} gives the following one-dimensional eigenvalue problem:
\begin{align}\label{1 dimen eigen}
    \begin{cases}
   u''(r)+\big((n-1)\coth{r}+h'(r)\big)u'(r)+\lambda_1(B(R), \beta) u(r)=0, \quad r\in (0, R),\\
   u'(0)=0,\quad u'(R)=-\beta u(R).
\end{cases}
\end{align}
Define $v(r):=-\frac{u'(r)}{u(r)}$. Direct computation gives
\begin{align}
    v'(r) = -\left((n-1)\coth{r}+h'(r)\right)v(r) + \lambda_1(B(R), \beta) + v(r)^2,
\end{align}
and 
\begin{align}\label{v second diff equation}
    v''(r) = 2v(r)v'(r) - \big((n-1)\coth{r}+h'(r)\big)v'(r) - \left(h''(r)-\frac{n-1}{\sinh^{2}{r}}\right)v(r).
\end{align}

\begin{prop}\label{prop:Hyperbolic_v_max}
Let $\beta > 0$  and let \(h\) be smooth, even, and convex. Suppose that $h(r)$ satisfies one of the following two conditions: 
\begin{enumerate}
    \item[\textup{(i)}] $\beta > \left(\frac{C_{2}(R)}{2}-\min\left\{1,\frac{n-1}{2}\right\}\right)\tanh R$, where $C_{2}(R)=\sup_{r\in (0,R]}\left(\frac{h'(r)}{\tanh r}+\frac{h''(r)}{\tanh' r}\right)$;
    \item[\textup{(ii)}] $0\leq h''(r) \leq  \frac{n-1}{\sinh^{2}r}$ for all $r\in(0,R)$.
\end{enumerate}
Then
\begin{equation}\label{eq:v-Hyperbolic-bound}
	0\leq v(r)<v(R)=\beta
	\qquad(0\leq r<R).
\end{equation}
\end{prop}

\begin{proof}
As in the Euclidean case, evenness and convexity imply
\(h'(0)=0\), \(h'\geq0\), and \(h''\geq0\).  

Under condition~\textup{(ii)}, set
\[
B(r):=\frac{n-1}{\sinh^2r}-h''(r)\geq0.
\]
For \(w=v'\), equation
\eqref{v second diff equation} reads
\[
w'=\left(2v-(n-1)\coth r-h'(r)\right)w+Bv.
\]
Given \(r>0\), choose \(a\in(0,r)\) so small that \(w(a)>0\), which
is possible by \eqref{eq:v-origin}.  Variation of constants yields
\[
w(r)
=
e^{\int_a^r(2v-A_0)\dd\tau}
\left[
w(a)+
\int_a^r
e^{-\int_a^s(2v-A_0)\dd\tau}B(s)v(s)\dd s
\right]>0.
\]
Thus \(v\) is strictly increasing, proving
\eqref{eq:v-Hyperbolic-bound}.

Now we turn to case (i): $\beta > \left(\frac{C_{2}(R)}{2}-\min\left\{1,\frac{n-1}{2}\right\}\right)\tanh R$. Where 
$$C_{2}(R)=\sup_{r\in (0,R]}\left(\frac{h'(r)}{\tanh r}+\frac{h''(r)}{\tanh' r}\right).$$
Let 
\begin{align*}
   k (R, n):=\frac{C_2(R)}{2}-\min\left\{1,\frac{n-1}{2}\right\}
\end{align*}
and 
$$\eta(r) := v(r) - k(R, n)\tanh r.$$ 
Substituting $v(r) = \eta(r) + k(R, n)\tanh r$ into \eqref{v second diff equation} we obtain
\begin{align*}
    0 &= \eta'' + \big((n-1)\coth r+h'(r)-2v\big)\eta' + \left(h''-\frac{n-1}{\sinh^{2}r}-2k(R, n)\tanh' r\right)\eta \\
    &\quad + k(R, n)\tanh'' r + k(R, n)\big((n-1)\coth r+h'(r)\big)\tanh' r\\
    &\quad + k(R, n)\left(h''-\frac{n-1}{\sinh^{2}r}\right)\tanh r - 2k(R, n)^{2}\tanh r\tanh' r.
\end{align*}
Direct calculation gives
\begin{align*}
    &k\tanh'' r + k\big((n-1)\coth r+h'(r)\big)\tanh' r + k\left(h''-\frac{n-1}{\sinh^{2}r}\right)\tanh r - 2k^{2}\tanh r\tanh' r\\
    =&-k\tanh r\tanh' r\left(2k-\frac{h'}{\tanh r}-\frac{h''}{\tanh' r}+2\right),
\end{align*}
then the differential equation for $\eta$ takes the form:
\begin{align}\label{eq:xi Hn}
    0 &= \eta'' + \big((n-1)\coth r+h'(r)-2v\big)\eta' + \left(h''-\frac{n-1}{\sinh^{2}r}-2k(R, n)\tanh' r\right)\eta \nonumber \\
    &\quad - k(R, n)\tanh r\tanh' r\left(2k(R, n)-\frac{h'}{\tanh r}-\frac{h''}{\tanh' r}+2\right).
\end{align}

% Given that $h\in C^{2}([0,\infty))$ and $h'(0)=0$, L'Hôpital's rule ensures the limits are well-defined at $r=0$. Thus, there exists a finite constant $C_{2}(R) > 0$ defined by
% \begin{equation*}
%     C_{2}(R) = \max_{r\in [0,R]}\left(\frac{h'(r)}{\tanh r}+\frac{h''(r)}{\tanh' r}\right).
% \end{equation*}
% Let us choose $k = \frac{C_{2}(R)}{2}-\min\left\{1,\frac{n-1}{2}\right\}$. 
Assuming $k > 0$ for the moment, a direct computation shows that the coefficient of $\eta$ in \eqref{eq:xi Hn} satisfies:
\begin{align}\label{4.12}
    h''-\frac{n-1}{\sinh^{2}r}-2k\tanh' r &= -\tanh' r\left(2k+(n-1)\coth^{2}r-\frac{h''}{\tanh' r}\right) \\
    &\le -\tanh' r\left(2k+(n-1)-\frac{h''}{\tanh' r}\right) \le 0,\nonumber
\end{align}
where we used  $\coth^2 r > 1$ for $r > 0$, and the definition of $C_2(R)$, which implies $$2k(R, n) + (n-1) \ge C_2(R) \ge \frac{h''(r)}{\tanh' r}.$$

From the boundary condition $v(R) = \beta$ and the assumption $\beta > k(R, n) \tanh R$, we  have
$$\eta(R) = \beta - k(R, n)\tanh R > 0, \quad \text{and}\quad \eta(0) = 0.$$ 
Thus, $\eta$ must attain a positive global maximum on $[0, R]$. If this maximum were attained at an interior point $r_{0}\in (0,R)$, then $\eta'(r_{0})=0$, $\eta''(r_{0})\leq 0$, and $\eta(r_0) > 0$. Thus we conclude from \eqref{eq:xi Hn} that
\begin{align*}
0\le& - \eta''(r_0)= \left(h''(r_0)-\frac{n-1}{\sinh^{2}r_0}-2k\tanh' r_0\right)\eta(r_0)\\ 
    &-k\tanh r_0\tanh' r_0\left(2k-\frac{h'(r_0)}{\tanh r_0}-\frac{h''(r_0)}{\tanh' r_0}+2\right)\\
     \le & -k\tanh r_0\tanh' r_0\left(2k-\frac{h'(r_0)}{\tanh r_0}-\frac{h''(r_0)}{\tanh' r_0}+2\right)\\
     < & 0,
\end{align*}
which is a contradiction. Where  we used  \eqref{4.12} in the second inequality, and the definition of $k(R, n)$ in the last inequality. 
Therefore, the maximum of $\eta$ must occur at $r=R$. Since $\eta(0)=0$ and $\eta(R)>0$, the maximum is attained at $r=R$, so $\eta(r) < \eta(R)$ for all $r \in [0, R)$. Because $\tanh r$ is strictly increasing on $(0, R)$, it follows that for any $r \in (0, R)$,
\begin{equation*}
    v(r) = \eta(r) + k\tanh r < \eta(R) + k\tanh R = v(R).
\end{equation*}

Finally, if $k(R,n)\le 0$, then $C_2(R)\le 2 \min \left\{1,\frac{n-1}{2}\right\}$, and hence
$$
\frac{h''(r)}{\tanh' r}\leq C_2(R) \le \min\{2, n-1\} \le n-1.
$$
Thus
\begin{equation*}
    h''(r) \leq \frac{n-1}{\cosh^{2}r} < \frac{n-1}{\sinh^{2}r},
\end{equation*}
so condition  (ii) holds. The previous argument applies, completing the proof. 
\end{proof}

\begin{corollary}\label{Monotonicity of v}
Under the first set of conditions in Theorem \ref{Euclidean log convex} and Theorem \ref{Hyperbolic space with log convex}, if the function 
$h''(r)-\frac{n-1}{S(r)^{2}}$ is strictly monotonically increasing on \((0,R)\), then \(v'(r)>0\) for all \(r\in (0,R)\).
\end{corollary}

\begin{proof}
We present the details for the Euclidean case only; the hyperbolic case is analogous and will be omitted.

 Let $\xi(r)$ be the auxiliary function defined in the proof of Proposition \ref{prop:Euclidean_v_max}. From Proposition  \ref{prop:Euclidean_v_max} we have  
 $\xi'(R)\ge 0$.
A direct computation yields
 $$
 \xi'(R)=v'(R)-\frac 1 2 C_{1}R,
 $$
and hence $$v'(R)=\xi(R)+\frac 1 2 C_1 R>0.$$
Now define 
\begin{align*}
f(r):=r^{n-1}e^{h(r)}u^{2}(r)v'(r),\quad r\in [0,R].
\end{align*}
A direct differentiation, using equation \eqref{v first derivative}, gives 
\begin{align}\label{eq:v(r) fundamental equation}
f'(r)
&=r^{n-1}e^{h(r)}u^{2}(r)\left(v''+\left(\frac{n-1}{r}+h'(r)+\frac{2u'}{u}\right)v'\right)\nonumber\\
&=r^{n-1}e^{h(r)}u^{2}(r)\left(\frac{n-1}{r^2}-h''(r)\right)v(r).
\end{align}
By assumption \(\frac{n-1}{r^2}-h''(r)\) is monotonically decreasing on \((0,R)\). Hence it has at most one zero. Therefore, only two cases are possible:
\begin{enumerate}
\item \(\frac{n-1}{r^2}-h''(r)\geq 0\) for all \(r\in (0,R)\);
\item there exists some \(r_0\in(0,R)\) such that \(\frac{n-1}{{r^2}}-h''(r)>0\) for \(r\in(0,r_0)\) and \(\frac{n-1}{{r^2}}-h''(r)<0\) for \(r\in(r_0,R)\).
\end{enumerate}
In either case, equation \eqref{eq:v(r) fundamental equation} implies $f(r)$  attains  its minimum on the boundary of  $(0,R)$.  Since $f(0)=0$ and $$f(R)=R^{n-1}e^{h(R)}u^{2}(R)v'(R)>0,$$ we conclude that  \(f(r)>0\) for all \(r\in(0,R)\). 
Therefore $v'(r)>0$ for all \(r\in(0,R)\).

% First we note that $0\leq \xi^{'}(R)=v^{'}(R)-C_{1-k}(R)(\frac{S_{k}(r)}{C_{k}(r)})^{'}$. So $v^{'}(R)>0$.
% We start from the following fundamental differential identity for \(v(r)\):
% \begin{align}\label{eq:v(r) fundamental equation}
% \big(S_{k}^{n-1}(r)e^{h(r)}u^{2}(r)v'(r)\big)'
% &=S_{k}^{n-1}(r)e^{h(r)}u^{2}(r)\left(v''+\left(\frac{(n-1)C_{k}(r)}{S_{k}(r)}+h'(r)+\frac{2u'}{u}\right)v'\right)\nonumber\\
% &=S_{k}^{n-1}(r)e^{h(r)}u^{2}(r)\left(\frac{n-1}{S_{k}^{2}(r)}-h''(r)\right)v(r).
% \end{align}

% Define the auxiliary function
% \begin{align*}
% f(r):=S_{k}^{n-1}(r)e^{h(r)}u^{2}(r)v'(r),\quad r\in(0,R).
% \end{align*}
% By assumption \(\frac{n-1}{S_{k}^{2}(r)}-h''(r)\) is monotonically decreasing on \((0,R)\). Consequently, the function \(\frac{n-1}{S_{k}^{2}(r)}-h''(r)\) admits at most one zero point in \((0,R)\), and only two possible scenarios can occur:
% \begin{enumerate}
% \item \(\frac{n-1}{S_{k}^{2}(r)}-h''(r)\geq 0\) for all \(r\in (0,R)\);
% \item there exists some \(r_0\in(0,R)\) such that \(\frac{n-1}{S_{k}^{2}(r)}-h''(r)>0\) for \(r\in(0,r_0)\) and \(\frac{n-1}{S_{k}^{2}(r)}-h''(r)<0\) for \(r\in(r_0,R)\).
% \end{enumerate}

% In both cases, the monotonicity property guarantees that \(f(r)\) attains its minimum on the boundary of the interval \((0,R)\).  Since $f(0)=0$ and $f(R)=S_{k}^{n-1}(R)e^{h(R)}u^{2}(R)v^{'}(R)>0$,  \(f(r)>0\) for all \(r\in(0,R)\).

% Note that \(S_{k}^{n-1}(r)\), \(e^{h(r)}\) and \(u^{2}(r)\) are all strictly positive on \((0,R)\). Therefore, the positivity of \(f(r)\) immediately implies \(v'(r)>0\) for all \(r\in(0,R)\).
\end{proof}

\section{Proofs of the Theorems \ref{Euclidean log convex} and \ref{Hyperbolic space with log convex}}\label{sect5}

Let $\Omega \subset \R^n$ or $\mathbb{H}^n$ be a bounded Lipschitz domain, and let $B(R)$ denote the geodesic ball of radius $R>0$ centered at the origin $o$ such that the weighted volumes satisfy 
 $|\Omega|_\gamma = |B(R)|_\gamma$.
Let $u$ be the positive first eigenfunction of \eqref{4.1} on $B(R)$. As established in Section \ref{sect4}, we define the radial function
\begin{align*}
v(r) := -\frac{u'(r)}{u(r)}, \quad r \in [0, R].
\end{align*}
The Robin boundary condition on $B(R)$ gives $v(R) = \beta$. Furthermore, Propositions~\ref{prop:Euclidean_v_max} and \ref{prop:Hyperbolic_v_max} give that 
$$0 \le v(r) < \beta \quad \text{for all}\quad  r \in [0, R).$$
We  now define the rearranged profile $\varphi_*$ on $B(R)$ by setting $\varphi_*(x) := v(r(x))$ for $x \in B(R)$.
Let $\psi$ denote the positive first eigenfunction on $\Omega$, normalized such that $\|\psi\|_{L^\infty(\Omega)}=1$.
For $t\in(0,1)$, consider the superlevel set $U(t):=\{y\in\Omega:\psi(y)>t\}$, and define the radius $r(t)$ implicitly via the volume-preserving condition $|B(r(t))|_\gamma=|U(t)|_\gamma$. Since $|U(t)|_\gamma$ is strictly decreasing in $t$, $r(t)$ is a strictly decreasing function mapping $(0,1)$ onto $(0,R)$. We then construct the test function $\varphi$ on $\Omega$  via level-set rearrangement:
\begin{align*}
\varphi(x) := v\big(r(\psi(x))\big), \qquad x\in\Omega.
\end{align*}

\begin{lemma}\label{phi inequality beta}
    The function $\varphi:\Omega\rightarrow \mathbb{R}$ defined above is measurable and satisfies $0\leq \varphi(x) < \beta$ for all $x \in \Omega$.
\end{lemma}

\begin{proof}
    Since $\psi$ is continuous on $\overline{\Omega}$, the weighted volume  $|U(t)|_\gamma$ is strictly decreasing, i.e., $|U(s)|_\gamma < |U(t)|_\gamma$ for $0\leq t<s\leq 1$. Consequently, $r(t)$ is strictly decreasing and hence measurable. As $u(r)$ is smooth and strictly positive on $[0, R]$, the function $v(r)$ is smooth on $[0,R)$,  and thus measurable. By definition, the sublevel set of $\varphi$ is given by
    \begin{align*}
       \Big \{x\in \Omega : \varphi(x) \leq c\Big\} = \psi^{-1}\Big(r^{-1}\big(v^{-1}((-\infty,c])\big)\Big),
    \end{align*}
    which is  measurable.  Finally, since $0 \le v(r) < v(R) = \beta$ for all $r \in [0,R)$, and since $r(\psi(x)) \in [0, R)$ for all $x \in \Omega$, it follows immediately that $0 \le \varphi(x) < \beta$.
\end{proof}

As in \cite[Lemma 5.2]{BucurCV}, we recall the following standard inequality.
\begin{lemma}\label{lem:perimeter_decomposition}
For almost every $t \in (0,1)$, the weighted perimeter has 
\begin{align*}
\operatorname{P}_{\gamma}(\partial U(t)) \leq  \operatorname{P}_{\gamma}(S(t)) + \operatorname{P}_{\gamma}(\Gamma(t)),
\end{align*}
where $\operatorname{P}_{\gamma}(E) := \int_E e^{h(r(x))} \, d\sigma$ denotes the weighted perimeter.
\end{lemma}

\begin{lemma}\label{lem:Rayleigh_quotient_inequality}
Let $\varphi$ be the test function defined above. Then, for almost every $t \in (0,1)$, we have
\begin{equation}\label{eq:Rayleigh_step}
H_{\Omega}(U(t),\varphi) \geq H_{B(R)}\big(B(r(t)),\varphi_{*}\big) = \lambda_{1}(B(R),\beta).
\end{equation}
Moreover, equality holds if and only if $U(t)$ is a geodesic ball centered at the origin and $\operatorname{P}_{\gamma}(\Gamma(t))=0$.
\end{lemma}

\begin{proof}
We estimate the boundary and interior terms separately.

For the boundary term, since $\varphi_*$ is constant on $\partial B(r(t))$ with value $v(r(t))$, we obtain, using the weighted isoperimetric inequality (Theorem~\ref{thm:weighted-isoperimetry}) and Lemma~\ref{lem:perimeter_decomposition},
\begin{align*}
\int_{\partial B(r(t))}\varphi_{*}e^{h}\,d\sigma
&= v\big(r(t)\big) \operatorname{P}_\gamma(\partial B(r(t))) \\
&\leq v\big(r(t)\big) \operatorname{P}_\gamma(\partial U(t)) \\
&\leq v\big(r(t)\big) \left( \operatorname{P}_\gamma(S(t)) + \operatorname{P}_\gamma(\Gamma(t)) \right) \\
&= \int_{S(t)} v(r(t)) e^{h}\,d\sigma + v(r(t)) \operatorname{P}_\gamma(\Gamma(t)) \\
&\leq \int_{S(t)}\varphi \, e^{h}\,d\sigma + \beta \operatorname{P}_\gamma(\Gamma(t)).
\end{align*}
Here, the last inequality uses  $\varphi(x) = v(r(t))$ on $S(t)$, and the strict bound $v(r(t)) < \beta$ from Lemma~\ref{phi inequality beta}.

For the interior $L^{2}$ term, the co-area formula gives
\begin{align*}
|U(t)|_{\gamma} = \int_{t}^{1}\int_{S(\tau)}\frac{1}{|\nabla \psi|}e^{h}\,d\sigma\,d\tau,
\end{align*}
and hence
\begin{align*}
\frac{d}{dt}|U(t)|_{\gamma} = -\int_{S(t)}\frac{1}{|\nabla \psi|}e^{h}\,d\sigma.
\end{align*}
Similarly, for the ball $B(r(t))$,
\begin{align*}
\frac{d}{dt}|B(r(t))|_{\gamma} = r'(t)\int_{\partial B(r(t))}e^{h}\,d\sigma.
\end{align*}
Since $|U(t)|_{\gamma}=|B(r(t))|_{\gamma}$, differentiating both sides yields
\begin{equation*}
-\int_{S(t)}\frac{1}{|\nabla \psi|}e^{h}\,d\sigma = r'(t)\int_{\partial B(r(t))}e^{h}\,d\sigma.
\end{equation*}
Using the co-area formula again and the fact that $\varphi$ is constant on each level set $S(\tau)$ with value $v(r(\tau))$, we compute
\begin{align*}
\int_{U(t)}\varphi^{2}(x)e^{h}\,d\mu
&= \int_{t}^{1} v\big(r(\tau)\big)^{2} \int_{S(\tau)}\frac{1}{|\nabla\psi|}e^{h}\,d\sigma\,d\tau \\
&= -\int_{t}^{1} v\big(r(\tau)\big)^{2} r'(\tau)\int_{\partial B_{r(\tau)}}e^{h}\,d\sigma\,d\tau \\
&= \int_{0}^{r(t)} v(s)^{2}\int_{\partial B(s)}e^{h}\,d\sigma\,ds \\
&= \int_{B(r(t))}\varphi_{*}^{2}e^{h}\,d\mu,
\end{align*}
where we applied the change of variables $s = r(\tau)$ (noting that $r'(\tau) < 0$). 
Combining the boundary estimate and the $L^{2}$ norm equality yields $$H_{\Omega}(U(t),\varphi) \geq H_{B(R)}\big(B(r(t)),\varphi_{*}\big) = \lambda_{1}(B(R),\beta).$$

For the equality case, if $U(t)$ is a geodesic ball centered at the origin and $\operatorname{P}_{\gamma}(\Gamma(t))=0$, then all the inequalities in the boundary estimate  become equalities, so equality holds in \eqref{eq:Rayleigh_step}. Conversely, if equality holds in \eqref{eq:Rayleigh_step}, then the weighted isoperimetric inequality applied to $\partial U(t)$ must be an equality, forcing $U(t)$ is a geodesic ball centered at the origin. Moreover, equality in the final step of the boundary estimate requires 
\begin{align*}
\int_{\Gamma(t)} \big(\beta - v(r(t))\big) e^{h}\,d\sigma = 0.
\end{align*}
Since $v(r(t)) < \beta$ by Lemma~\ref{phi inequality beta}, this implies $\operatorname{P}_{\gamma}(\Gamma(t))=0$.
\end{proof}

We are now ready to prove Theorems~\ref{Euclidean log convex} and \ref{Hyperbolic space with log convex}.

\begin{proof}[Proof of Theorems~\ref{Euclidean log convex} and \ref{Hyperbolic space with log convex}]
By Proposition~\ref{lambda representation 3.1} and Proposition~\ref{lambda low bound}, we have $\lambda_{1}(\Omega,\beta) \geq H_{\Omega}(U(t),\varphi)$ for almost every $t \in (0,1)$. Combining this with Lemma~\ref{lem:Rayleigh_quotient_inequality}, immediately yields the desired Faber-Krahn  inequality $$\lambda_{1}(\Omega,\beta) \geq \lambda_{1}(B(R),\beta).$$ 
Suppose now that equality holds, i.e., $\lambda_{1}(\Omega,\beta) = \lambda_{1}(B(R),\beta)$. The equality case in Theorem~\ref{lambda low bound} implies that $\varphi(x) = \frac{|\nabla \psi(x)|}{\psi(x)}$ for almost every $x \in \Omega$. Consequently, by Proposition~\ref{lambda representation 3.1}, we deduce that 
\begin{align*}
\lambda_{1}(\Omega,\beta) = H_{\Omega}(U(t),\varphi) = \lambda_{1}(B(R),\beta)
\end{align*} 
for almost every $t \in (0,1)$. This proves
\eqref{eq:Euclidean-FK} and \eqref{eq:Hyperbolic-FK}.

Applying the equality condition in Lemma~\ref{lem:Rayleigh_quotient_inequality}, we conclude that $U(t)$ is a geodesic ball centered at the origin for almost every $t \in (0,1)$. Since the superlevel sets $U(t)$ are nested and all are balls sharing the same center, their union $\Omega = \bigcup_{t \in (0,1)} U(t)$ coincides with a geodesic ball centered at the origin up to a set of measure zero. Finally, because $\Omega$ is a Lipschitz domain, it is uniquely determined by its interior up to a null set, which implies that $\Omega$ itself is exactly a geodesic ball centered at the origin.

Suppose now that equality holds. Applying the equality condition in Lemma~\ref{lem:Rayleigh_quotient_inequality}, so
\[
\varphi=\frac{|\nabla\psi|}{\psi}
\quad\text{almost everywhere in }\Omega.
\]
Consequently equality holds in
\eqref{eq:Rayleigh_step} for almost every admissible level and \(U_t\) is a weighted
isoperimetric region for almost every such \(t\).

Choose admissible levels \(t_j\downarrow0\) so that the sets
\(U_{t_j}\) are nested and isoperimetric.  Since \(\psi>0\) almost
everywhere,
\[
U_{t_j}\uparrow\Omega
\quad\text{up to a null set},\qquad
r(t_j)\uparrow R.
\]
In the hyperbolic case, Theorem~\ref{thm:weighted-isoperimetry}
\textup{(b)} gives
\[
U_{t_j}=B(r(t_j))
\quad\text{up to null sets}.
\]
Taking the increasing union yields
\(\Omega=B(R)\) up to a null set.

In the Euclidean case, Theorem~\ref{thm:weighted-isoperimetry}
\textup{(a)} says that every
\(U_{t_j}\) is either centered or is a Euclidean ball contained in
the flat core.  If the radii eventually exceed the flat-core radius,
only the centered alternative remains, and the increasing union is
\(B_o(R)\).  Otherwise the \(U_{t_j}\) form an increasing sequence
of Euclidean balls contained in \(B_o(R_0)\).  Write them as
\(B_{c_j}(a_j)\).  Inclusion gives
\[
|c_{j+1}-c_j|\leq a_{j+1}-a_j.
\]
Hence \(c_j\) converges, \(a_j\) increases to a limit \(a\), and
their union is the Euclidean ball \(B_c(a)\), still contained in the
flat core.  This proves the asserted Euclidean classification.

Conversely, a centered ball gives equality trivially.  If
\(\Omega=B_c(a)\) lies in the Euclidean flat core, then the density is
constant on both \(\Omega\) and the centered ball \(B_o(R)\), the
weighted-volume condition gives \(R=a\), and both Robin problems
reduce to the same constant-density Euclidean ball problem.
Therefore equality also holds for every ball in the flat core.
\end{proof}

\section{A lower bound for the second Robin eigenvalue}\label{sect6}
Let $\psi$ denote an eigenfunction associated with the second eigenvalue $\l_2(\Omega, \beta)$. Since $\psi$ must change sign in $\Omega$, we define its positive and negative nodal domains by
\begin{align*}
\Omega^{+}:=\{x\in \Omega: \psi(x)>0\},\qquad \Omega^{-}:=\{x\in \Omega: \psi(x)<0\},
\end{align*}
and its positive and negative parts  by
\begin{align*}
\psi^{+}(x):=\max\{\psi(x),0\},\qquad \psi^{-}(x):=\max\{-\psi(x),0\}.
\end{align*}
Then $\psi^{\pm}\in H^{1}(\Omega)\cap C(\overline{\Omega})$, and $\nabla\psi^{\pm}\not\equiv 0$. Let $B^{+}$ and $B^{-}$ be balls centered at the origin satisfying the weighted volume matching conditions 
\begin{align*}
    |B^{+}|_{\gamma}=|\Omega^{+}|_{\gamma}, \qquad |B^{-}|_{\gamma}=|\Omega^{-}|_{\gamma}.
\end{align*}
For any $\epsilon>0$, define the boundary strip by
\begin{align*}
S_{\epsilon}:=\{x\in \Omega: \operatorname{dist}(x,\partial \Omega)<\delta\},
\end{align*}
where $\delta=\delta(\epsilon)>0$ is chosen such that $|S_{\epsilon}|_{\gamma}<\epsilon$. Define the enlarged domain
\begin{align*}
U_{\epsilon}:=\Omega^{+}\cup S_{\epsilon}.
\end{align*}
Note that $\partial \Omega\subset \partial U_{\epsilon}$. Let $\Gamma_{\epsilon}:=\partial U_{\epsilon}\setminus\partial \Omega$; the set $\Gamma_{\epsilon}$ is compactly embedded in $\Omega$. 

Let $H_{U_{\epsilon}}$ denote the closure,  with respect to the $H^{1}(U_{\epsilon})$-norm, of the space  of all $C^{\infty}(\overline{\Omega})$ functions whose supports are compactly contained in $U_{\epsilon}\cup \partial\Omega$. Define
\begin{align}
    \Lambda_{1}(U_{\epsilon},\beta)=\inf_{\varphi\in H(U_{\epsilon})}\frac{\int_{U_{\epsilon}}|\nabla \varphi|^{2}e^{h(|x|)}\,d\mu+\beta\int_{\partial \Omega}\varphi^{2}e^{h(|x|)}\,d\sigma}{\int_{U_{\epsilon}}\varphi^{2}e^{h(|x|)}\,d\mu}
\end{align}
\begin{lemma}
   For every $\epsilon>0$, we have 
   \begin{align}\label{6.2}
    \lambda_{2}(\Omega,\beta)\geq \Lambda_{1}(U_{\epsilon},\beta) . 
   \end{align}
\end{lemma}

\begin{proof} 
Since $\psi$ is an eigenfunction for  the second eigenvalue, it satisfies 
\begin{align*}
    \int_{\Omega}\nabla\psi \cdot \nabla \varphi e^{h(|x|)}\,d\mu+\beta\int_{\partial \Omega}\psi \varphi e^{h(|x|)}\,d\sigma=\lambda_{2}(\Omega,\beta)\int_{\Omega}\psi(x)\varphi(x)e^{h(|x|)}\,d\mu
\end{align*}
for all $\varphi\in H^{1}(\Omega)$.
Taking $\varphi=\psi^{+}$, we obtain
\begin{align*}
     \int_{\Omega}|\nabla \psi^{+}|^{2}e^{h(|x|)}\,d\mu+\beta\int_{\partial \Omega}\psi^{+}(x)^{2} e^{h(|x|)}\,d\sigma=\lambda_{2}(\Omega,\beta)\int_{\Omega}\psi^{+}(x)^{2}e^{h(|x|)}\,d\mu.
\end{align*}
Since $\int_{\Omega}\psi^{+}(x)^{2}e^{h(|x|)}\,d\mu\neq 0$, it follows that
\begin{align*}
    \l_{2}(\Omega,\beta)=\frac{\int_{U_{\epsilon}}|\nabla \psi^{+}|^{2}e^{h(|x|)}\,d\mu+\beta \int_{\partial \Omega}\psi^{+}(x)^{2}e^{h(|x|)}\,d\sigma}{\int_{U_{\epsilon}}\psi^{+}(x)^{2}e^{h(|x|)}\,d\mu}
\end{align*}
Note that $\psi^{+}(x)=0$ on $\Gamma_{\epsilon}$. As shown in \cite{KennegyPAMS}, we have $\psi^{+}\in H_{U_{\epsilon}}$, and hence
$$
\lambda_{2}(\Omega,\beta)\geq \Lambda_{1}(U_{\epsilon},\beta).
$$
\end{proof}

For a bounded domain $V\subset B(R)$, there exist positive constants $c_1(R)$, $c_2(R)$ such that 
\begin{align*}
 \frac{1}{c_{2}(R)}|V|_{\gamma}\leq |V|\leq c_{1}(R)|V|_{\gamma}.  
\end{align*}
 Hence by the arguments in \cite[Lemma 3.3]{KennegyPAMS}, we have the following approximation result.
\begin{lemma} [\cite{KennegyPAMS}]
    There exists a sequence of Lipschitz domains $U^m\subset U_{\epsilon}$ such that
    \begin{itemize}
        \item [(1)]    $\partial \Omega \subset \partial U^m$ ;
        \item [(2)] $\operatorname{dist}(\partial \Omega,\partial U^m\setminus \partial \Omega)>0$ for all $m$;
        \item [(3)] $|U_{\epsilon}\setminus U^m|_{\gamma}\rightarrow 0$;
        \item [(4)] $\Lambda_{1}(U^m,\beta)\rightarrow \Lambda_{1}(U_{\epsilon},\beta)$ as $m\rightarrow \infty$.
    \end{itemize}
\end{lemma}
Since $\partial U^m$ is Lipschitz, the applicable one of Theorems \ref{Euclidean log convex} and  \ref{Hyperbolic space with log convex} yields
$$\lambda_{1}(U^m,\beta)\geq \lambda_{1}(B^m,\beta)$$ for all $m$, where $B^m$ is a ball centered at origin with $|U^m|_{\gamma}=|B^m|_{\gamma}$. Moreover, Since $H_{U^m}\subset H^{1}(U^m)$, we have $\Lambda_{1}(U^m,\beta)\geq \lambda_{1}(U^m,\beta)$. Let $B_{\epsilon}$ be a geodesic ball centered at origin with $|B_{\epsilon}|_{\gamma}=|U_{\epsilon}|_{\gamma}$. Since $|U^m|_{\gamma}\rightarrow |U_\epsilon|_{\gamma}$, we have $|B^m|_{\gamma}\rightarrow |B_{\epsilon}|_{\gamma}$, and by continuity of $\lambda_1(B_r, \beta)$  with respect to the radius $r$, 
 $\lambda_{1}(B^m,\beta)\rightarrow \lambda_{1}(B_{\epsilon},\beta)$. Therefore,
\begin{align}
    \Lambda_{1}(U_{\epsilon},\beta)=\lim_{m\rightarrow\infty}\Lambda_{1}(U^m,\beta)\geq \limsup_{m\rightarrow \infty}\lambda_{1}(U^m,\beta)\geq \lim_{m\rightarrow \infty} \lambda_{1}(B^m,\beta)=\lambda_{1}(B_{\epsilon},\beta).
\end{align}
Recall from \eqref{6.2} that  $\lambda_{2}(\Omega,\beta)\geq \Lambda_{1}(U_{\epsilon},\beta)$ for all $\epsilon>0$, and  $|B_{\epsilon}|_{\gamma}\rightarrow |B^{+}|_{\gamma}$, we obtain $$\lambda_{2}(\Omega,\beta)\geq \lambda_{1}(B^{+},\beta).$$ 
By applying the same argument to $\Omega^-$, we get
 $$\lambda_{2}(\Omega,\beta)\geq \lambda_{1}(B^{-},\beta).$$
 Hence
\begin{align}\label{6.4}
 \lambda_{2}(\Omega,\beta)\geq \max\{\lambda_{1}(B^{+},\beta),\lambda_{1}(B^{-},\beta)\},   
\end{align}
where $|B^{\pm}|_{\gamma}=|\Omega^{\pm}|_{\gamma}$.

\begin{lemma}\label{lamba monotonicity}
    The eigenvalue $\lambda_{1}(B(r),\beta)$ is strictly decreasing as a function of $r$.
\end{lemma}

\begin{proof}
    Let $0<r_{1}<r_{2}$, and let $w(r)$ denote the positive eigenfunction associated with $\lambda_{1}(B(r_{2}),\beta)$, i.e., 
    \begin{align*}
    \begin{cases}
       w''(r)+\left(\dfrac{(n-1)C(r)}{S(r)}+h'(r)\right)w'(r)+\lambda_{1}(B(r_{2}),\beta)w(r)=0, & r\in (0,r_{2}),\\
        w'(0)=0,\quad w'(r_{2})+\beta w(r_{2})=0.
    \end{cases}
    \end{align*}
    Restricting $w(r)$ ($r\in [0, r_2]$) to $B(r_1)$, define 
    \begin{align*}
        \beta(r_{1}):=-\dfrac{w'(r_{1})}{w(r_{1})}.
    \end{align*}
  Then  it follows that
    \begin{align*}
        \lambda_{1}(B(r_{2}),\beta)=\lambda_{1}\big(B(r_{1}),\beta(r_{1})\big).
    \end{align*}
    By Corollary \ref{Monotonicity of v}, we have $\beta(r_{1})\leq \beta$. Since $\lambda_{1}(B(r),\beta)$ is monotonically increasing in the boundary parameter $\beta$. we conclude that
    \begin{align*}
        \lambda_{1}(B(r_{2}),\beta)
        =\lambda_{1}\big(B(r_{1}),\beta(r_{1})\big)
        \leq \lambda_{1}(B(r_{1}),\beta).
    \end{align*}
    Thus, $\lambda_{1}(B(r),\beta)$ is decreasing in $r>0$.
\end{proof}

\begin{proof}[Proof of Theorem \ref{thm3}]
  It follows from Lemma \ref{lamba monotonicity} that
\begin{align*}
    \max\big\{\lambda_{1}(B^{+},\beta),\lambda_{1}(B^{-},\beta)\big\}\geq \lambda_{1}(D,\beta),
\end{align*}
where $D$ denotes the ball centered at the origin satisfying  half the weighted volume  condition $|D|_{\gamma}=\frac{1}{2}|\Omega|_{\gamma}$. Combining this with \eqref{6.4} yields the desired inequality
$$ \lambda_2(\Omega,\beta)\ge \lambda_1(D,\beta).$$
This completes the proof. 
\end{proof}
\subsection*{Acknowledgements}
The authors thank Professor Frank Morgan for drawing their attention
to the radial log-convex density conjecture attributed to Kenneth Brakke.

 \bibliographystyle{plain}
	\bibliography{ref}

\end{document}